\documentclass[12pt, reqno]{amsart}
\usepackage{amssymb, hyperref}
\usepackage{amsthm}
\usepackage{graphicx}
\usepackage[all,2cell]{xy}
\usepackage{verbatim}
\usepackage{tikz}
\usepackage{tikz-cd}
\usepackage{hyperref}
\usetikzlibrary{matrix,arrows,decorations.pathmorphing}
\usepackage{mathrsfs}
\numberwithin{equation}{section}
\newtheorem*{theorem*}{Theorem}
\newtheorem*{corollary*}{\bf Corollary}
\usepackage{lipsum}
\usetikzlibrary{matrix,arrows,decorations.pathmorphing}
\usepackage{lineno}
\newtheorem{theorem}{Theorem}[section]
\newtheorem{corollary}[theorem]{Corollary}
\newtheorem{definition}[theorem]{Definition}
\newtheorem{example}[theorem]{Example}
\newtheorem{lemma}[theorem]{Lemma}

\newtheorem{proposition}[theorem]{Proposition}
\newtheorem{remark}[theorem]{Remark}
\usepackage{csquotes}
\usepackage{epigraph}

\title[A note on toric degeneration of a BSDH variety]{A note on toric degeneration of a Bott-Samelson-Demazure-Hansen variety}
 \author[B.N. Chary]{B. Narasimha Chary}
\address{%
B. Narasimha Chary\\
Institut Fourier, UMR 5582 du CNRS\\
Universit{\'e} de Grenoble Alpes\\
 CS 40700, 38058\\
Grenoble cedex 09, France.\\
Email: narasimha-chary.bonala@univ-grenoble-aples.fr
}

\curraddr{
Fakult\"{a}t f\"{u}r Mathematik\\
Ruhr-Universit\"{a}t Bochum, Bochum\\
Germany\\
Email: Narasimha.bonala@rub.de
}

\begin{document}
 
\maketitle

\begin{abstract} In this paper we study the geometry of toric degeneration 
of a Bott-Samelson-Demazure-Hansen (BSDH) variety, which was algebraically constructed in \cite{pasquier2010vanishing} and \cite{parameswaran2016toric}.
We give some applications to BSDH varieties.
Precisely, we classify Fano, weak Fano and log Fano BSDH varieties and their toric limits in Kac-Moody setting. 
We prove some vanishing theorems for the cohomology of tangent bundle (and line bundles) on BSDH varieties.
We also recover the results in \cite{parameswaran2016toric}, by toric methods.
\end{abstract}
\let\thefootnote\relax\footnotetext{The author was supported by AGIR Pole MSTIC project run by the University of Grenoble Alpes, France.} 

\section{Introduction}\label{intro}

Bott-Samelson-Demazure-Hansen (for short, BSDH) varieties are natural desingularizations of Schubert varieties
in the flag varieties. They were first introduced by Raoul Bott and Hans Samelson in a differential geometric and topological context (see \cite{bott1958}).  These were algebraically constructed by 
M. Demazure and H.C. Hansen independently by 
adapting a differential geometric approach from the paper
of  Bott and Samelson (see \cite{demazure1974} and \cite{hansen}). These are also known as Bott-Samelson varieties.  In \cite{brion}, these varieties are named as Bott-Samelson-Demazure-Hansen (BSDH) varieties when we work in the setting of algebraic geometry. 
Briefly, the BSDH varieties are iterated projective line bundles, given by factoring the 
Schubert variety using Bruhat decomposition.
These varieties  
depend on the given expression of the Weyl group element corresponding to the
Schubert variety (see for instance \cite[Page 32]{Chary1}). 
We also see in this paper some properties of these varieties which depend on
the given expression.

In \cite{grossberg1994bott}, M. Grossberg and Y. Karshon constructed  toric degenerations of BSDH varieties
by complex geometric methods.
In \cite{pasquier2010vanishing} B. Pasquier and in 
\cite{parameswaran2016toric} 
A.J. Parameswaran and P. Karuppuchamy constructed these toric degenerations algebraically. 
B. Pasquier used 
these  degenerations to obtain vanishing results for the cohomology of line bundles on BSDH varieties (see \cite{pasquier2010vanishing}). 
When $G$ is a simple algebraic group and the expression $\tilde w$ is reduced,
Fanoness and weak Fanoness of the BSDH variety $Z(\tilde w)$ are considered in \cite{Charyfano}.
Here we consider the case of Kac-Moody setting and the expression $\tilde w$ is not necessarily reduced.
In \cite{parameswaran2016toric}, the authors got a Fanoness criterion of the 
limiting toric variety for a simple simply connected algebraic group 
by geometric methods. 
In this paper we study the limiting toric variety of a BSDH variety in more detail by methods of toric geometry
and we prove some applications to BSDH varieties.
We also recover the results in \cite{parameswaran2016toric} and extend them to the Kac-Moody setting.
The key idea for many results in this article is that the toric limit is a \textquoteleft Bott tower'. 
These are studied in \cite{Moriconeoftowers} and some of their properties can be 
transferred to BSDH varieties by using the semi-continuity theorem.

Let $G$ be a Kac-Moody group over the field of complex numbers
(for the definition 
see \cite{kumar2012kac}).
Let $B$ be a Borel subgroup containing a fixed 
maximal torus $T$. Let $W$ be the Weyl group corresponding to the pair $(G, B, T)$ and let $w\in W$. Let $\tilde w:=s_{\beta_1}\cdots s_{\beta_n}$ be an 
expression (possibly non-reduced) of $w$ in simple reflections and let $Z(\tilde w)$ be the BSDH variety corresponding to $\tilde w$ (see Section 2). Let 
$Y_{\tilde w}$ be the toric limit of $Z(\tilde w)$ constructed as in \cite{pasquier2010vanishing} or \cite{parameswaran2016toric}
(see Section \ref{degeneration}).
We see that $Y_{\tilde w}$ is a Bott tower, the iterated $\mathbb P^1$-bundle over a point $\{pt\}$ where each $\mathbb P^1$-bundle is the projectivization of a rank $2$ decomposable vector bundle
(see Corollary \ref{tl}).
We prove that the ample cone $Amp(Y_{\tilde w})$ of $Y_{\tilde w}$ can be identified with a subcone of the ample cone $Amp(Z(\tilde w))$ of $Z(\tilde w)$
(see Corollary \ref{amplecone}).

 Recall that a smooth projective variety $X$ is called {\it Fano} (respectively, {\it weak Fano}) if its 
 anti-canonical divisor $- K_X$ is 
 ample (respectively, {\it nef} and {\it big}).
Following \cite{anderson2014schubert}, we say  that a pair $(X, D)$ of a normal projective variety $X$ and an 
 effective $\mathbb Q$-divisor $D$ is {\it log Fano} if it is Kawamata log terminal and $-(K_X+D)$ is ample.

 Let $\tilde w=s_{\beta_1} \cdots s_{\beta_i}\cdots s_{\beta_j} \cdots s_{\beta_r}$ be an expression (remember that $\beta_k$'s are simple roots).
Let $\beta_{ij}:=\langle \beta_j, \check \beta_i \rangle$, where $\check \beta_i$ is the co-root of $\beta_i$.
 Now we define some conditions on the expression $\tilde w$ 
  (see \cite[Section 1]{Moriconeoftowers} and \cite{correctionMoriconeoftowers}).
    Define for $1\leq i \leq r$,  
 \[\eta^+_i:=\{r\geq j> i: \beta_{ij}>0\}~~\mbox{and}~~ \eta^-_i:=\{r\geq j> i: \beta_{ij}<0\}.\]
 If $|\eta_{i}^+|=1$ (respectively, $|\eta_{i}^+|=2$), then let $\eta_{i}^+=\{m\}$ (respectively, $\eta_{i}^+=\{m_1, m_2\}$).
 If $|\eta_{i}^-|=1$ (respectively, $|\eta_{i}^-|=2$),
 then set $\eta^-_i=\{l\}$ (respectively,  $\eta^-_i=\{l_1, l_2\}$).
 Note that since $\beta_i'$s are simple roots, if $\beta_{ij}>0$ then $\beta_i=\beta_j$ and $\beta_{ij}=2$.
 \noindent 
 \begin{itemize}
 \item $N^{1}_i$  is the condition that  
 
 (i) $|\eta^+_i|=0$, $|\eta^-_i|\leq 1$, and if $|\eta^-_i|=1$ 
 then $\beta_{il}=-1$.
  \item $N^{2}_i$  is the condition that: 
  
   (i) $|\eta^+_i|=0$, $|\eta_i^-|\leq 2$, and 
  if $|\eta^-_i|=1 (\mbox{respectively,}~~|\eta^-_i|=2)$ then  $\beta_{li}=-1$ or $-2$ 
  (respectively, $\beta_{il_1}=-1=\beta_{il_2}$).
       \end{itemize}

   \begin{definition}   
 We say the expression $\tilde w$ satisfies \underline{condition $I$} (respectively,  \underline{condition $II$})
 if
 $N^1_{i}$ (respectively,  $N^2_i$) holds for all $1\leq i\leq r$. 
 Note that $N_i^1\Longrightarrow N_i^2 $ for all $1\leq i\leq r$.
   \end{definition}
 We refer to Section \ref{fanoweakfanologfano} for examples on the conditions $I$ and $II$. The main result of the paper is the following: 
 \begin{theorem*}[See Lemma \ref{toriclimitfano} and Theorem \ref{fanoforbsdh}]\
  \begin{enumerate}
  \item  
  If $\tilde w$ satisfies 
  $I$, then $Y_{\tilde w}$ and  $Z(\tilde w)$ are Fano.
  \item If 
  $\tilde w$ satisfies 
  $II$, then $Y_{\tilde w}$ and $Z(\tilde w)$ are weak Fano.
\end{enumerate}
 \end{theorem*}

  In \cite{Chary1} and \cite{Chary11}, we have obtained some vanishing results for 
 the cohomology of tangent bundle of $Z(\tilde w)$, when $G$ is finite dimensional and $\tilde w$ is reduced 
(see \cite[Section 3]{Chary1} and \cite[Theorem 8.1]{Chary11} ).
  Here we get some vanishing results in Kac-Moody setting. Let $T_{Z(\tilde w)}$ denote the tangent bundle of $Z(\tilde w)$.
 \begin{corollary*}[see Corollary \ref{vanishing1}]\
  If $\tilde w$ satisfies 
 $I$, then  $H^i(Z(\tilde w), T_{Z(\tilde w)})=0$ for all $i\geq 1$. In particular,
   $Z(\tilde w)$ is locally rigid.
\end{corollary*}

 In \cite{anderson2014schubert}, D. Anderson and
 A. Stapledon studied
the log Fanoness of Schubert varieties,
and in \cite{anderson2014effective}, log Fanoness of BSDH varieties is
studied for chosen divisors. 
Let  $D$ be a divisor in $Z(\tilde w)$ with support in the boundary of $Z(\tilde w)$.
For $1\leq i \leq r$, we define some constants $f_i$  which again depend on the given expression $\tilde w$
(for more details see Section \ref{fanoweakfanologfano}).
 \begin{corollary*}[see Corollary \ref{logfano02}]
 The pair $(Z(\tilde w), D)$ is log Fano if   $f_i>0$ for all $1\leq i\leq r$.
    \end{corollary*}

The article is organized as follows: In section \ref{preliminaries}, we recall the construction of BSDH varieties. 
In Section \ref{degeneration}, we give the algebraic construction of toric degeneration of a BSDH variety. In Section \ref{botttowers}, we describe the limiting toric variety
 as an iterated $\mathbb P^1$-bundle. 
In Section \ref{linebundles on BSDH}, we see some vanishing results of cohomology of  line bundles on BSDH varieties.
Section \ref{fanoweakfanologfano} contains the results on Fano, weak Fano and log Fano properties of BSDH varieties and their toric limits.
We also study the vanishing results on cohomology of tangent bundle on BSDH varieties. 
In Section \ref{AJPPK}, we recover the results in \cite{parameswaran2016toric} by toric methods.

\section{Preliminaries}\label{preliminaries} In this section we 
recall the construction of Bott-Samelson-Demazure-Hansen varieties (see \cite{brion} and \cite{kumar2012kac}) and we recall some definitions in toric geometry which are used in this article (for more details on toric varieties see \cite{cox2011toric} and also \cite{fulton}).
We work over the field of complex numbers throughout.

\subsection{BSDH varieties}
Let $A=(a_{ij})_{1\leq i,j\leq n}$ be a generalized Cartan matrix. 
Let $G$ be the Kac-Moody group associated to $A$ (see \cite[Chapter IV]{kumar2012kac}).
Fix a maximal torus $T$ and a Borel subgroup $B$ containing $T$. Let
$S:=\{\alpha_1,\ldots, \alpha_n\}$
be the set of all simple roots of $(G, B, T)$.  We denote $s_{\alpha_i}$ the simple reflection corresponding to $\alpha_i$.
Note that the Weyl group $W$ of $G$ is generated by 
$$\{s_{\alpha_i}: 1\leq i \leq n\}.$$

Let $w\in W$, an expression $\tilde w$ of $w$ is a sequence $(s_{\beta_1}, \ldots, s_{\beta_r})$ of simple reflections 
$s_{\beta_1}, \ldots, s_{\beta_r}$ such that $w=s_{\beta_1}\cdots s_{\beta_r}$. 
An expression $\tilde w$ of $w$ is said to be reduced if the number $r$ of simple reflections is minimal.
In such case we call $r$ the length of $w$. 
By abuse of notation, we also denote the expression $\tilde w$ by $\tilde w=s_{\beta_1}\cdots s_{\beta_r}$.
For $\alpha\in S$,
we denote $P_{\alpha}$, the minimal parabolic subgroup of $G$
generated by $B$ and a representative of $s_{\alpha}$. 
\begin{definition}
 Let $w\in W$ and $\tilde{w} :=s_{\beta_1}\cdots s_{\beta_r}$ be an expression (not necessarily reduced) of $w$.
The Bott-Samelson-Demazure-Hansen (for short, BSDH) variety corresponding to $\tilde w$ is 
$$Z(\tilde w):=P_{\beta_1}\times \cdots \times P_{\beta_r}/B^r,$$
where the action of $B^r$ on $P_{\beta_1}\times \cdots \times P_{\beta_r}$ is defined by 
$$(p_1,\ldots, p_r)\cdot (b_1,\ldots, b_r)=(p_1b_1, b_1^{-1}p_2b_2, \ldots ,b_{r-1}^{-1}p_rb_r) ~~\mbox{for all}~~ p_i\in P_{\beta_i}, b_i\in B.$$
 
 \end{definition}
 
 These are smooth projective varieties of dimension $r$. The BSDH variety can be described also as an
iterated projective line bundle, where each projective bundle 
is the projectivization of certain rank $2$ vector bundle (not necessarily decomposable).
 There is a natural morphism $\phi_{\tilde w}: Z(\tilde w)\longrightarrow G/B$ defined by $$[(p_1, \ldots, p_r)]\mapsto p_1\cdots p_rB.$$

 If $\tilde w$ is reduced, the BSDH variety $Z(\tilde w)$ is a 
 natural desingularization of the Schubert variety, 
 the $B$-orbit closure of $wB/B$ in $G/B$
(see \cite{demazure1974}, \cite{hansen} and \cite[Chapter VIII]{kumar2012kac}).
We can also construct the BSDH variety as an iterated $\mathbb P^1$-bundles.
Let $\tilde w':=s_{\beta_1}\cdots s_{\beta_{r-1}}$. 
Let $f: G/B\longrightarrow G/P_{\beta_r}$ be the map given by $gB\mapsto gP_{\beta_r}$
and let $p:Z(\tilde w') \longrightarrow G/P_{\beta_r}$ be the map given by
$[(p_1,\ldots , p_{r-1})]\mapsto p_1\cdots p_{r-1}P_{\beta_r}$.
Then we have the following cartesian diagram (see \cite[Page 66]{brion} and \cite[Chapter VII]{kumar2012kac}):
\[
 \xymatrix{Z(\tilde w)=Z(\tilde w')\times_{G/P_{\beta_r}}G/B  \ar[rrr]^{\phi_{\tilde w}} \ar[d]^{f_{\tilde w}} &&& G/B \ar[d]^{f}\\
 Z(\tilde w')\ar[rrr]_{p} &&& G/P_{\beta_r}}
\]

Note that $f_{\tilde w}$ is a $\mathbb P^1$-fibration and the relative tangent 
bundle $T_{f_{\tilde w}}$ of $f_{\tilde w}$ is $\phi^{*}_{\tilde w}(\mathcal L_{\beta_r})$, 
where $\mathcal L_{\beta_r}$ is the homogeneous line bundle on $G/B$ corresponding to $\beta_r$. Using the cohomology of the relative tangent bundle $T_{f_{\tilde w}}$
we studied the cohomology of the tangent bundle of $Z(\tilde w)$, when $G$ 
is finite dimensional
and $\tilde w$ is a reduced expression (see \cite{Chary1} and \cite{Chary11}).
The fibration $f_{\tilde w}$ comes with a natural section $\sigma_{\tilde w}:Z(\tilde w')\to Z(\tilde w)$ induced
by the projection 
$$P_{\beta_1} \times \cdots \times P_{\beta_r}\to P_{\beta_1}\times \cdots \times P_{\beta_{r-1}}.$$
For the toric limits we get two natural sections, as will be explained in Section \ref{degeneration}.
For all $i\in \{1, \ldots, r\}$, we denote $Z_i$, the divisor in $Z(\tilde w)$ defined by 
$$\{[(p_1,\ldots, p_r)]\in Z(\tilde w): p_i\in B\}.$$
In \cite {lauritzen2002line}, N. Lauritzen and J.F. Thomsen proved that $Z_i's$
forms a basis of the Picard group of $Z(\tilde w)$ and they also proved that if $\tilde w$ is a reduced expression 
these form a basis of the monoid of effective divisors (see \cite[Proposition 3.5] {lauritzen2002line}).  Recently, 
the effective divisors of $Z(\tilde w)$ for $\tilde w$ non-reduced case have been considered in \cite{anderson2014effective}.

\subsection{Toric varieties}
\begin{definition}
 A normal variety $X$ is called a \it{toric variety} (of dimension $n$)
 if it contains an $n$-dimensional
 torus $T$ (i.e. $T=(\mathbb C^{*})^n$) as a Zariski open subset such that the action of 
 the torus on itself by multiplication extends to an action of the torus on $X$.
\end{definition}
Toric varieties are completely described by the combinatorics of the corresponding fans. 
We denote the fan corresponding to a toric variety by $\Sigma$ and the collection of cones of dimension $s$ in $\Sigma$ by
$\Sigma(s)$ for $1\leq s\leq n$.
For each cone $\sigma \in \Sigma$, we denote $V(\sigma)$, the orbit closure of the orbit corresponding to cone $\sigma$.
For each $\sigma\in \Sigma$, $\sigma(1):=\sigma\cap \Sigma(1)$. For each $\rho\in \Sigma(1)$, we can associate a divisor in $X$, we denote it by $D_{\rho}$ (see \cite[Chapter 4]{cox2011toric} for more details). 
We recall the following:
\begin{definition}\label{def} \

 \begin{enumerate}
  \item 
We say
   $P \subset \Sigma (1)$ is a {\it  primitive collection} if $P$ is not contained in 
 $ \sigma(1)$ for some $ \sigma  \in  \Sigma$
but any proper subset is. Note that if $ \Sigma$ is simplicial, 
primitive collection means that $P$ does not generate a cone in $ \Sigma$ but every proper subset does.
\item  Let $P= \{ \rho_1,  \ldots,  \rho_k \}$ be a primitive collection in a complete simplicial fan $ \Sigma$.
  Denote by $u_{\rho}$, 
  the primitive vector of the ray $\rho\in \Sigma$.
  Then $\sum_{i=1}^{k}u_{ \rho_i}$ is zero or in the relative interior of a cone $\gamma_{P}$ in $\Sigma$ with a unique expression 
 \begin{eqnarray}\label{4.1}
   \sum_{i=1}^{k}u_{ \rho_i}-(\sum_{  \rho  \in \gamma_{P}(1)}c_{ \rho}u_{ \rho})=0 .
   \end{eqnarray}
 where $c_{ \rho} \in  \mathbb Q_{>0}.$ Then we call (\ref{4.1}) the {\it primitive relation} of $X$ corresponding to $P.$ 
\item For a primitive relation $P$, we can associate an element $r(P)$ in $N_1(X)$, where $N_1(X)$ is the real vector space of numerical classes of 
one-cycles in $X$  (see \cite[Page 305]{cox2011toric}).
 
 \end{enumerate}

 \end{definition}

\section{Toric degeneration of a BSDH variety}\label{degeneration}
In \cite{grossberg1994bott},  toric degenerations of  BSDH varieties were constructed  
by  complex geometric methods.
In \cite{pasquier2010vanishing} and \cite{parameswaran2016toric} 
they have given an algebraic 
construction for toric degeneration of a BSDH variety. 
We recall the algebraic construction here.

 Note that the simple roots are linearly independent elements in the character group of $G$. Let $N$ be the lattice of one-parameter subgroups of $T$. We can choose a positive integer
$q$ and an injective morphism $\lambda:\mathbb G_m\longrightarrow T$ (i.e. $\lambda \in N$ and $\lambda$ is injective)
such that for all $1\leq i\leq n$ and $u\in \mathbb G_m$, $\alpha_i(\lambda(u))=u^q$ (see \cite[Page 2836]{pasquier2010vanishing}).
When $G$ is finite dimensional, for each one-parameter subgroup $\lambda \in N$, define
$$P(\lambda):=\{g\in G: lim_{u\to 0} \lambda(u)g\lambda(u)^{-1} ~\mbox{exists in }~ G\}.$$
The set $P(\lambda)$ is a parabolic subgroup and the unipotent radical $R_u(P(\lambda))$ of $P(\lambda)$ is given by 
$$R_u(P(\lambda))=\{g\in G:lim_{u\to 0} \lambda(u)g\lambda(u)^{-1} ~\mbox{is identity in}~ G\}.$$
Any parabolic subgroup of $G$ is of this form (see \cite[Proposition 8.4.5]{springer2010linear}).
 Choose a one-parameter subgroup $\lambda \in N$
such that the corresponding parabolic subgroup is $B$. 
Let us define an endomorphism of  $G$ for all $u\in \mathbb G_m$ by 
$$\tilde {\Psi}_{u}:G\to G, ~\mbox{ }~ g\mapsto \lambda(u)g\lambda(u)^{-1}.$$
Let $\mathcal B$ be the set of all endomorphisms of $B$. 
Now define a morphism 
$$\Psi: \mathbb G_m \to \mathcal B ~~\mbox{ by}~~ u \mapsto \tilde \Psi_u|_{B}.$$
This map can be extended to $0$ and for all $x\in U$, $\Psi_u|_{B}(x)$ goes to identity when 
$u$ goes to zero. Let $\mathbb A^1:=S   pec\mathbb C[t]$ be the affine line over $\mathbb C$. We denote for all $u\in \mathbb A^1$,  $\Psi_{u}$ the image of $u$ in $\mathcal B$.
Note that $\Psi_u$ is the identity on $T$ and  $\Psi_0$ is the projection from $B$ to $T$.
Let $\tilde w = s_{\beta_1}\cdots s_{\beta_r}$ be an expression.

\begin{definition}\

(i) Let $\mathcal X$ be the variety defined by
$$
 \mathcal X:=\mathbb A^1\times P_{\beta_1}\times \cdots \times P_{\beta_r}/B^r,
 $$ 
  where the action of $B^r$ on $\mathbb A^1 \times P_{\beta_1}\times \cdots \times P_{\beta_r}$ is given by
  $$(u, p_1,\ldots, p_r)\cdot (b_1,\ldots, b_r)=(u, p_1b_1, \Psi_u(b_1)^{-1}p_2b_2, \ldots, \Psi_{u}(b_{r-1})^{-1}p_rb_r).$$
 (ii) For all $i\in \{1,\ldots, r\}$, we denote $\mathcal Z_{i}$ the divisor in $\mathcal X$ defined by 
  $$\{(u,p_1,\ldots, p_r)\in \mathcal Z: p_i\in B \}.$$
 \end{definition}

 Note that $\mathcal X$ and $\mathcal Z_{i}'s$ are integral.
 Let $\pi: \mathcal X\to \mathbb A^1$ be the projection onto the first factor. 
   Then we have the following theorem (see \cite[Proposition 1.3 and 1.4]{pasquier2010vanishing} and \cite[Theorem 9]{parameswaran2016toric} ).
 
 \begin{theorem}\label{theorem1}\
 
 \begin{enumerate}
 
  \item $\pi:\mathcal X \to \mathbb A^1$ is a smooth projective morphism.
  \item For all $u\in \mathbb A^1\setminus \{0\}$, the fiber $\pi^{-1}(u)$ is isomorphic to the BSDH variety $Z(\tilde w)$
such that $\pi^{-1}(u)\cap \mathcal Z_i$ corresponds to the divisor $Z_{i}$ in $Z(\tilde w)$.
\item $\pi^{-1}(0)$ is a smooth projective toric variety.
 \end{enumerate}
 
 \end{theorem}
 
 We denote  $\mathcal X_u:=\pi^{-1}(u)$ for $u\in \mathbb A^1$ and the limiting toric variety $\mathcal X_0=\pi^{-1}(0)$ by $Y_{\tilde w}$.

 \section{Connection to Bott towers}\label{botttowers}
In this section we describe the toric limit $Y_{\tilde w}$ as an iterated $\mathbb P^1$-bundle. We also recall some results on Bott towers from \cite{Moriconeoftowers}.
 Let $\{e_1^+,\ldots, e_{r}^{+}\}$ be the standard basis of the lattice $\mathbb Z^r$.
Define for all $i\in\{1,\ldots, r\}$, 
\begin{equation}\label{e--}
 e_i^{-}:=-e_i^{+}-\sum_{j>i}\beta_{ij}e_{j}^{+},
\end{equation}
where  
$\beta_{ij}:=\langle \beta_j, \check {\beta_i} \rangle$.
 The following proposition will give the description of the fan of the toric variety 
$Y_{\tilde w}$ (see \cite[Proposition 1.4]{pasquier2010vanishing}).
\begin{proposition}\label{fan107}\

\begin{enumerate}
 \item 
The fan $\Sigma$ of the smooth toric variety $Y_{\tilde w}$ consists of the 
cones generated
by subsets of 
$$\{e_{1}^{+}, \ldots, e_{r}^{+}, e_{1}^-,\ldots, e_{r}^-\}$$
and
containing no subset of the form 
$\{e_{i}^+, e_{i}^-\}$. 

 \item For all $i\in \{1,\ldots, r\}, \mathcal Z^0_{i}$ is the 
 irreducible $(\mathbb C^*)^r$-stable divisor in
 $Y_{\tilde w}$ corresponding to the one-dimensional cone of $\Sigma$ generated by $e_i^{+}$ and 
 these form a basis of the divisor class group of $Y_{\tilde w}$.
\end{enumerate}
\end{proposition}

Note that the maximal cones of $\Sigma$ are generated by 
$\{e_i^{\epsilon}: 1\leq i\leq r, \epsilon \in \{+, -\}\}$ .
We denote the divisor corresponding to the one-dimensional cone $\rho_i^{\epsilon}$ generated by $e_i^{\epsilon}$ by $D_{\rho_i^{\epsilon}}$ for $\epsilon \in \{+, -\}$.
Let $\tilde w':=s_{\beta_1}\cdots s_{\beta_{r-1}}$.  Then we get a toric morphism 
$f_r:Y_{\tilde w} \to Y_{\tilde w'}$
induced by the lattice 
map 
$\overline f_r:\mathbb Z^r\to \mathbb Z^{r-1},$
the projection onto the first $r-1$ coordinates.
We prove,
\begin{lemma}\label{f_r}\

\begin{enumerate}
 \item 
 $f_{r}:Y_{\tilde w} \to Y_{\tilde w'}$ is a toric $\mathbb P^1$-fibration
 with two disjoint toric sections.
 \item $Y_{\tilde w}\simeq \mathbb P(\mathcal O_{Y_{\tilde w'}}\oplus \mathscr L)$ for some unique line bundle $\mathscr L$ on $Y_{\tilde w'}$.
  \end{enumerate}
 
\end{lemma}

\begin{proof}
Let $\Sigma'$ be the fan corresponding to the toric variety $Y_{\tilde w'}$. 
From the above proposition, we can see that 
$\Sigma$ has a splitting by $\Sigma'$ and $\{e_r^+, 0, e_r^{-}\}$. Then by \cite[Theorem  3.3.19]{cox2011toric}, 
$$f_r:Y_{\tilde w} \to Y_{\tilde w'}$$ 
is a locally trivial
fibration 
with the fan  $\Sigma_F$ of the fiber being $\{e_r^+, 0, e_r^{-}\}$. 
Since $\Sigma_{F}$ is the fan of the projective line $\mathbb P^1$, we conclude  
$f_r$ is a toric $\mathbb P^1$-fibration.
  As toric sections of the toric fibration correspond to the maximal cones in $\Sigma_F$, 
we get two  disjoint toric sections for $f_r$. This proves (1).

Proof of (2):
Since  $f_r: Y_{\tilde w}\to Y_{\tilde w'}$ is $\mathbb P^1$-fibration with a section,
we see $Y_{\tilde w}$ is a projective bundle $\mathbb P(\mathscr E)$ over $Y_{\tilde w'}$
corresponding to a 
rank 2 vector bundle $\mathscr E$ on $Y_{\tilde w'}$ (see for example \cite[Chapter V, Proposition 2.2, page 370]{hartshorne}).

Recall that the sections of projective bundle $Y_{\tilde w}=\mathbb P(\mathscr E)$
correspond to the quotient line bundles of $\mathscr E$ (see \cite[Proposition 7.12]{hartshorne}).
Since $Y_{\tilde w}=\mathbb P(\mathscr E)$ is projective line bundle on $Y_{\tilde w'}$ with
two disjoint sections, we see  $\mathscr E$ is decomposable as a direct sum of line bundles on $Y_{\tilde w'}$. 

As 
$$\mathbb P(\mathscr E)\simeq \mathbb P(\mathscr L'\otimes \mathscr E)$$
for any line bundle $\mathscr L'$ on $Y_{\tilde w'}$ (see \cite[Lemma 7.9]{hartshorne}), 
we can assume without loss of generality  
$$\mathscr E=\mathcal O_{Y_{\tilde w'}}\oplus \mathscr L$$
for some unique line bundle $\mathscr L$ on $Y_{\tilde w'}$.
Hence $Y_{\tilde w}\simeq \mathbb P(\mathcal O_{Y_{\tilde w'}}\oplus \mathscr L)$ and this completes the proof of the lemma.
\end{proof}

\begin{definition}
 A Bott tower of height $r$  
is a sequence of projective bundles 
$$Y_r  \overset{{\pi_r}}\longrightarrow Y_{r-1} \overset{\pi_{r-1}}\longrightarrow \cdots \overset{\pi_2}\longrightarrow  Y_1=\mathbb P^1 \overset{\pi_1} \longrightarrow Y_0=\{pt\}, $$ 
where $Y_i=\mathbb P (\mathcal O_{Y_{i-1}}\oplus \mathcal L_{i-1})$ for a line bundle $\mathcal L_{i-1}$ over $Y_{i-1}$ for all $1\leq i\leq r$ and $\mathbb P(-)$ denotes the projectivization
(see for more detalis \cite{civan2005bott} and also \cite[Section 2]{Moriconeoftowers}).
\end{definition}

Then by definition of Bott tower and by Lemma \ref{f_r}(2) we get: 
\begin{corollary}\label{tl}
 The toric limit $Y_{\tilde w}$ is a Bott tower.
\end{corollary}
We have the following situation: 

 \[
\begin{tikzcd}
\mathbb P^1 \arrow[d] && \mathbb P^1 \arrow[d]\\
Z(\tilde w) \arrow[d, "f_{\tilde w}"] \arrow[rr, rightsquigarrow]
&& Y_{\tilde w}  \arrow[d, "f_r"]\\
 Z(\tilde w') \arrow[u,bend left=50, red] \arrow[rr, rightsquigarrow] && Y_{\tilde w'} \arrow[u,bend right=50, red]\arrow[u,bend right=75, red]
\end{tikzcd}
\]

Here the top arrows are inclusions of the fibres.
Recall that the Bott towers bijectively correspond to the upper triangular matrices with integer entries (see \cite[Section 3]{civan2005bott}).
Here the upper triangular matrix $M_{\tilde w}$ corresponding to $Y_{\tilde w}$ is given by 
$$M_{\tilde w}=\begin{bmatrix}
        1 & \beta_{12} & \beta_{13} & \dots & \beta_{1r}\\
        0 & 1 & \beta_{23} & \dots & \beta_{2r}\\
        0 & 0 & 1 &\dots & \beta_{3r}\\
        \vdots & \vdots & &\ddots & \vdots \\
        0 & \dots & \dots & & 1
       \end{bmatrix}_{r\times r},
$$ where $\beta_{ij}$'s are integers as defined before.
Let $P_i:=\{\rho_i^+, \rho_i^-\}$ for $1\leq i \leq r$. Then by \cite[Lemma 4.3]{Moriconeoftowers}, 
$\{P_i: 1\leq i\leq r\}$ is the set of all primitive collections of $Y_{\tilde w}$. 
For each $1\leq i \leq r$, we denote the cone in the definition of primitive relation (see Section \ref{preliminaries}) corresponding to $P_i$ by $\gamma_{P_i}$. 
Let $D=\sum_{\rho\in \Sigma(1)} a_{\rho}D_{\rho}$ be a toric divisor in $Y_{\tilde w}$ with $a_{\rho}\in \mathbb Z$ and 
 for $1\leq i\leq r $, define
$$d_i:=(a_{\rho_i^+}+a_{\rho_i^-}-\sum_{\gamma_j \in \gamma_{P_i}(1)}c_{j}a_{\gamma_j}).$$
Then we recall the following from \cite[Lemma 5.1]{Moriconeoftowers}:
\begin{lemma}\label{amplenef}\
 \begin{enumerate}
  \item  $D$ is ample if and only if $d_i>0$ for all $1\leq i \leq r$.
  \item $D$ is numerically effective ({\it nef}) if and only if 
  $d_i\geq 0$ for all $1\leq i \leq r$.
\end{enumerate}
\end{lemma}
Also note that the conditions $I$ and $II$ on $\tilde w$ are same as the conditions on $M_{\tilde w}$ as in \cite{Moriconeoftowers}.

\section{Vanishing results on Cohomology of certain line bundles on BSDH varieties}\label{linebundles on BSDH}
Let $X$ be a smooth projective variety.
Recall $N^1(X)$ denote the real finite dimensional vector space of numerical classes of 
real divisors in $X$ (see \cite[ \S 1, Chapter IV]{kleiman1966toward}).
The ample cone $Amp(X)$ of $X$ is the cone in $N^1(X)$ 
generated by classes of ample divisors. 

\subsection{Ample cone of the toric limit of BSDH variety}
In \cite{lauritzen2002line}, the ampleness of line bundles on BSDH variety $Z(\tilde w)$ is studied.
Now we compare the ample cone of the toric limit $Y_{\tilde w}$ with that of the BSDH-variety $Z(\tilde w)$ as a consequence of Theorem \ref{theorem1}.
\begin{corollary}\label{amplecone} The
 ample cone $Amp(Y_{\tilde w})$ of $Y_{\tilde w}$ can be identified with a subcone of the ample cone $Amp(Z(\tilde w))$  of $Z(\tilde w)$.
\end{corollary}
\begin{proof}
 By Theorem \ref{theorem1}, $\pi: \mathcal X \to \mathbb A^1$ is a smooth projective morphism with $\mathcal X_0=Y_{\tilde w}$ and 
 $\mathcal X_{u}=Z(\tilde w)$ for $u\neq 0$. 
  Let $\mathcal L=\{\mathcal L_u: u\in \mathbb A^1\}$ be a line bundle on $\pi: \mathcal X \to \mathbb A^1$ with $\mathcal L_0$ is an ample line bundle 
 on $Y_{\tilde w}$.
Note that the ampleness of line bundle is an open
 condition for the proper morphism $\pi$, i.e. 
 there exists an open subset $U$ in $\mathbb A^1$
 containing $0$ such that $\mathcal L_u$ is an ample line bundle on $\mathcal X_u$ for all $u\in U$  (see  \cite[Theorem 1.2.17]{lazarsfeld2004positivity}). 
 Hence we can identity 
 $Amp(Y_{\tilde w})$  with  a subcone of $Amp(Z(\tilde w))$.
\end{proof}
\subsection{Vanishing results}

In \cite{pasquier2010vanishing}, B. Pasquier obtained
vanishing theorems for the cohomology of any line bundle on BSDH varieties in some degrees 
(see \cite[Theorem 0.1]{pasquier2010vanishing}).
Here we obtain a class of line bundles on BSDH varieties for which all higher degree cohomologies vanish.  To state our result we need the following notation.
 
  Let $1\leq i \leq r$, define   $h_i^{i-1}:= -\beta_{(i-1)i}$ and  
$$h_i^j:=\begin{cases}
         0 & ~\mbox{for}~ j>i.\\
         1& ~\mbox{for}~ j=i.\\
         -\sum_{k=j}^{i-1}\beta_{k i}(h_k^j) & ~\mbox{for}~ j<i.
        \end{cases}
$$
Let $\epsilon \in \{+, -\}$. Define $\Sigma(1)^{\epsilon}:=\{\rho_i^{\epsilon}: 1\leq i \leq r\}.$
Then we can write a toric divisor $D$ in $Y_{\tilde w}$ as follows:
$$D=\sum_{\rho\in \Sigma(1)}a_{\rho}D_{\rho}=
\sum_{\rho\in \Sigma(1)^+}a_{\rho}D_{\rho}+
\sum_{\rho\in \Sigma(1)^-}a_{\rho}D_{\rho}.$$

For $1\leq i \leq r$, let 
$$g_i:=a_{\rho_i^+}+\sum_{j=i}^{r}a_{\rho_j^-}h_j^i.$$ Recall $d_i$ from Section \ref{botttowers},
$$d_i:=(a_{\rho_i^+}+a_{\rho_i^-}-\sum_{\gamma_j \in \gamma_{P_i}(1)}c_{j}a_{\gamma_j}),$$

Let $D'=\sum_{i=1}^rg_iZ_i$ be a divisor in $Z(\tilde w)$, where $Z_i$ is as in Section \ref{preliminaries} for $1\leq i \leq r$.
Now we prove the following:
\begin{lemma}
 If $d_i\geq 0$ for all $1\leq i \leq r$, then $H^j(Z(\tilde w), D')=0$ for all $j>0$.
\end{lemma}
\begin{proof}
 If $d_i\geq 0$ for all $1\leq i \leq r$, by Lemma \ref{amplenef}, $\sum_{\rho\in \Sigma(1)}a_{\rho}D_{\rho}$ is a {\it nef} divisor in $Y_{\tilde w}$.
Then we have 
\begin{equation}\label{bb}
 H^j(Y_{\tilde w}, \sum_{\rho\in \Sigma(1)}a_{\rho}D_{\rho})=0 ~\mbox{for all}~ j>0
 \end{equation}
(see \cite[Theorem 9.2.3, page 410]{cox2011toric} or \cite[Theorem 2.7, page 77]{oda}).
Recall that by Theorem \ref{theorem1}, we have $$\mathcal {Z}_i^{x}=Z_{i} ~\mbox{ for}~ 0\neq x\in k ~\mbox{ and}~
\mathcal Z_i^{0}=D_{\rho_i^+}.$$
By \cite[Corollary 3.3]{Moriconeoftowers}, we can write 
$$D=\sum_{\rho\in \Sigma(1)}a_{\rho}D_{\rho} \sim  \sum_{i=1}^{r}g_iD_{\rho_i^+}.$$
Hence by (\ref{bb}), Theorem \ref{theorem1} and by semi-continuity theorem (see \cite[Theorem 12.8]{hartshorne}), we get 
$$H^j(Z(\tilde w), D')=0 ~\mbox{ for all}~ j>0.$$
 \end{proof}
 \begin{remark}  Due to complexity of the conditions in \cite{pasquier2010vanishing},  it is not clear that the above result can be deduced from \cite[Theorem 0.1]{pasquier2010vanishing}.
 \end{remark}

\section{Fano, Weak Fano and log Fano BSDH varities}\label{fanoweakfanologfano}

 \subsection{Fano and weak Fano properties}\label{6.2BSDH} In this section, 
  we observe that Fano and weak Fano properties for BSDH variety $Z(\tilde w)$ depend on the given expression $\tilde w$.
 We use the terminology from Section \ref{intro}. First we discuss the conditions $I$ and $II$ with some examples. 
  We use the ordering of simple roots as in \cite[Page 58]{Hum1}. 
  Note that if $\beta_{ij}>0$ then $\beta_{ij}=2$ and so the conditions in \cite{correctionMoriconeoftowers} are reduces to our conditions $I$ and $II$.
  
 \underline{\bf The condition $I$:}
 \begin{example} Assume that
 $|\eta^+_i|=0$ and $|\eta^-_i|=0$.
  This condition means that the expression $\tilde w$ is fully commutative without repeating the simple reflections. 
  For example if $G=SL(n, \mathbb C)$ and $\tilde w=s_{\alpha_1}s_{\alpha_3} \cdots s_{\alpha_r}$, $1< r \leq n-1$ and $r$ is odd, then 
 $|\eta_i^+|=0$ and $|\eta_i^-|=0$ for all $i$.
 Hence
$\tilde w$
satisfies the condition $I$ and also observe that in this case we have 
$$Y_{\tilde w}\simeq Z(\tilde w)\simeq \mathbb P^1\times \cdots \times \mathbb P^1  \hspace {0.7cm} (dim(Z(\tilde w))~\mbox{times}~).$$

 \end{example}

 \begin{example}
 Let $G=SL(n, \mathbb C)$ and  fix $1\leq j < r\leq n-1$ such that $j$ is even and $r$ is odd. 
 Let $\tilde w=s_{\alpha_1}s_{\alpha_3}\cdots  s_{\alpha_{j-3}}s_{\alpha_{j-1}} s_{\alpha_j} s_{\alpha_{j+1}}s_{\alpha_{j+3}}  \cdots s_{\alpha_r}$.
 Note that $s_{\alpha_j}$ appears only once in the expression $\tilde w$ and $|\eta_i^+|=0$ for all $i$. 
 Let $p=j/2+1$ be the \textquoteleft position of $s_{\alpha_j}$' in the expression $\tilde w$, then $|\eta_i^-|=0$ for all $i\neq p , p-1$ and $|\eta_{p-1}^-|=1=|\eta_p^-|$ with 
 $\beta_{p-1 p}=-1=\beta_{p p+1}$.
Hence $\tilde w$ satisfies condition $I$. 
 
  \end{example}

 \underline{\bf The condition $II$:}
 
First observe that $\tilde w$ satisfies the condition $I$ then it also satisfies the condition $II$. 
 \begin{example}
 Let $G=SL(n, \mathbb C)$.
 \begin{enumerate}
  \item 
 Fix $1\leq j < r\leq n-1$ such that
 $j$ is even and $r$ is odd. 
 Let $\tilde w=s_{\alpha_1}s_{\alpha_3}\cdots  s_{\alpha_{j-3}}s_{\alpha_j}s_{\alpha_{j-1}} s_{\alpha_{j+1}}s_{\alpha_{j+3}}  \cdots s_{\alpha_r}$
 (observe that, here,  we interchanged $s_{\alpha_j}$ and 
 $s_{\alpha_{j-1}}$ in the example of condition $I$).
Then $|\eta_i^+|=0$ 
 and $|\eta_i^-|\leq 2$ for all $i$. Let $p=j/2-1$ be the \textquoteleft position of $s_{\alpha_j}$' in the expression $\tilde w$,
 then $|\eta_i^-|=0$ for all $i\neq p$ and $|\eta_p^-|=2$ with 
 $\beta_{p p+1}=-1=\beta_{p p+2}$. 
Hence 
$\tilde w$ satisfies the condition $II$ but not $I$.
 
\item Let $\tilde w=s_{\alpha_1}s_{\alpha_3}s_{\alpha_1}$. Then $|\eta_1^+|=1$ with 
 $\beta_{13}=2$, and $|\eta_1^-|=
 |\eta_2^+|=\eta_2^-|=0$ .
 Hence $\tilde w$
  does not satisfy $II$.

 \end{enumerate}

  \end{example}

 \begin{example}
 Observe that the condition  $|\eta_i^-|=1$ and $\beta_{il}=-2$, happens only in non-simply laced cases.
  Let $G=SO(5, k)$ (i.e. $G$ is of type $B_2$) and recall that we have  $\langle \alpha_{1}, \alpha_2 \rangle =-2$ and $\langle \alpha_{2}, \alpha_1 \rangle =-1$. 
  \begin{enumerate}
   \item 
  Let $\tilde w_1=s_{\alpha_2}s_{\alpha_1}$, then $\tilde w_1$
  satisfies $II$ but not $I$.
  \item Let  $\tilde w_2=s_{\alpha_1}s_{\alpha_2}$, then $\tilde w_2$
  satisfies $I$. 

  \end{enumerate}

  \end{example}
  
  \begin{example}

  Let $G$ be of type $G_2$, then we have  $\langle \alpha_{1}, \alpha_2 \rangle =-1$ and $\langle \alpha_{2}, \alpha_1 \rangle =-3$.
  \begin{enumerate}
   \item 
Let $\tilde w_1=s_{\alpha_2}s_{\alpha_1}$, then $\tilde w_1$
  satisfies $I$
\item Let $\tilde w_2=s_{\alpha_1}s_{\alpha_2}$, then
$\tilde w_2$
  does not satisfy $II$. 
  
  \end{enumerate}

  \end{example}

 \begin{example}\
 
  \begin{enumerate}
  
   \item Let $G$ be a Kac-Moody group with generalized Cartan matrix $
   \begin{bmatrix}
2 & -4 \\
-1 & 2
\end{bmatrix}.$\\
   (a) Let $\tilde w=s_{\alpha_1}s_{\alpha_2}s_{\alpha_1}$. Then we can see that $|\eta_1^+|=1$ with $\beta_{13}=2$ and so $\tilde w$ does not satisfy $II$.\\
   (b) Let $\tilde w=s_{\alpha_2}s_{\alpha_1}$. Then $|\eta_i^+|=0$ for $i=1, 2$ and $|\eta_1^-|=1$ with $\beta_{12}=-1$. Hence $\tilde w$ satisfies $I$.\\
   (c) Let $\tilde w=s_{\alpha_1}s_{\alpha_2}$. Then $|\eta_i^+|=0$ for $i=1, 2$ and $|\eta_1^-|=1$ with $\beta_{12}=-4$. Hence $\tilde w$ does not satisfy $II$.
   \item  Let $G$ be a Kac-Moody group with generalized Cartan matrix  $\begin{bmatrix}
2 & -2 \\
-2 & 2
\end{bmatrix}.$\\
    Let $\tilde w=s_{\alpha_1}s_{\alpha_2}$ or $\tilde w=s_{\alpha_2}s_{\alpha_1}$. Then  $|\eta_i^+|=0$  for $i=1, 2$ and $|\eta_2^-|=0$. Also note that  $|\eta_1^-|=1$ with $\beta_{12}=-2$. Hence $\tilde w$ satisfies $II$ but not $I$.\\
    
    \end{enumerate}
 \end{example}

  Now we have the following result:
 \begin{lemma}\label{toriclimitfano} \
 
 \begin{enumerate}
  \item 
   If $\tilde w$ satisfies $I$, then  $Y_{\tilde w}$ is Fano.
\item If $\tilde w$ satisfies $II$, then $Y_{\tilde w}$ is weak Fano.
 \end{enumerate}
 \end{lemma}
\begin{proof} This follows from Corollary \ref{f_r} and \cite[Theorem 6.3]{Moriconeoftowers} (see also \cite{correctionMoriconeoftowers}).
\end{proof}

Recall the following (see for instance \cite[Corollary 6.2]{Moriconeoftowers}):
\begin{lemma}\label{big}
 Let $X$ be a smooth projective variety and $D$ be an effective divisor. Let $supp(D)$ denote the support of $D$. If $X\setminus supp(D)$ is affine, then $D$ is big.
\end{lemma}
We prove the following:

 \begin{theorem}\label{fanoforbsdh}\
 
 \begin{enumerate}
  \item 
  If $\tilde w$ satisfies 
  $I$, then $Z(\tilde w)$ is Fano.
\item If 
$\tilde w$ satisfies 
$II$, then $Z(\tilde w)$ is weak Fano.
  \end{enumerate}

 \end{theorem}
\begin{proof}
 
First recall that the canonical line bundle 
$\mathcal O_{Z(\tilde w)}(K_{Z(\tilde w)})$ of $Z(\tilde w)$
is given by 
$$\mathcal O_{Z(\tilde w)}(K_{Z(\tilde w)})=\mathcal O_{Z(\tilde w)}(-\partial Z(\tilde w))\otimes \mathcal L_{\tilde w}(-\delta),$$
where $\partial Z(\tilde w)$ is the boundary divisor of 
$Z(\tilde w)$ and  $\delta \in N $
such that 
$\langle \delta, \check \alpha \rangle=1$ for all $\alpha\in S$, where $\check \alpha$ is the co-root of $\alpha$
(see \cite[Proposition 8.1.2]{kumar2012kac} and also \cite[Proposition 2]{ramanathan1985schubert}).
Note that if $G$ is finite dimensional, $\delta$ is half sum of the positive roots.

 By Theorem \ref{theorem1}, $\phi: \mathcal X \to \mathbb A^1$ is a smooth projective morphism 
 with $\mathcal X_0=Y_{\tilde w}$ and 
 $\mathcal X_{u}=Z(\tilde w)$ for $0\neq u\in \mathbb A^1$. 
  
Proof of (1):
 By \cite[Theorem 1.2.17]{lazarsfeld2004positivity}, if $-K_{\mathcal X_0}$ is ample then 
 $-K_{\mathcal X_u}$ is ample for  $u\neq 0$.
By Lemma \ref{toriclimitfano}, if $\tilde w$ satisfies 
$I$, then $-K_{Y_{\tilde w}}$ is ample. 
Hence we conclude that  if  $\tilde w$ satisfies 
$I$,
then $Z(\tilde w)$ is Fano.

Proof of (2): 
First we prove $-K_{Z(\tilde w)}$ is {\it big}.
 Let $$Z_0:=Z(\tilde w)\setminus \partial Z(\tilde w).$$ 
 Note that $Z_0$ is an open affine subset of $Z(\tilde w)$. 
 Then by Lemma \ref{big},
$\partial Z(\tilde w)$ is {\it big}. Since $\delta$ is a regular dominant integral weight, $\mathcal L_{\tilde w}(\delta)$ is {\it nef} (see Section \cite[Section 7.2]{kumar2012kac}).
Since $$\mathcal O(-K_{Z(\tilde w)})=\mathcal 
O(\partial Z(\tilde w))\otimes \mathcal L_{\tilde w}(\delta),$$ 
we conclude $-K_{Z(\tilde w)}$ is {\it big}, as 
tensor product of a {\it big} and a {\it nef} line bundles is again a
big line bundle.
 By \cite[Theorem 1.4.14]{lazarsfeld2004positivity} and $\mathcal X_{u}=Z(\tilde w)$ for $u\neq 0$,
we can see that 
if $-K_{\mathcal X_0}$ is {\it nef} then $-K_{\mathcal X_u}$ is also {\it nef} for $u\neq 0$.
  Therefore, (2) follows from Lemma \ref{toriclimitfano}(2).
\end{proof}
  
 There exists an expression $\tilde w$ such that the BSDH variety $Z(\tilde w)$ is Fano 
 (respectively, weak Fano) but the toric limit $Y_{\tilde w}$ is not Fano (respectively, not weak Fano). 
 \begin{example}
  Let $G=SL(5, \mathbb C)$. 

\begin{enumerate}
 \item 
Let $\tilde w=s_{\alpha_1}s_{\alpha_1}$. 
Then $Z(\tilde w)\simeq \mathbb P^1\times \mathbb P^1$, 
which is Fano. 
The toric limit $Y_{\tilde w}\simeq \mathbb P(\mathcal O_{\mathbb P^1}\oplus \mathcal O _{\mathbb P^1}(2))$. It is well known that $Y_{\tilde w}$ is not Fano (also see \cite{correctionMoriconeoftowers}).

\item Let $\tilde w=s_{\alpha_1}s_{\alpha_2}s_{\alpha_1}$.
Then it can be seen that $Z(\tilde w)$ is Fano (see \cite[Example 5.4]{Charyfano}).
By the Theorem in \cite{correctionMoriconeoftowers}, we can see that the toric limit $Y_{\tilde w}$ is weak Fano but not Fano.

\end{enumerate}

\end{example}
\begin{example}
 Let $G=SO(7, k)$, i.e. $G$ is of type $B_3$. Let $\tilde w=s_{\alpha_2}s_{\alpha_3}s_{\alpha_1}s_{\alpha_2}$.
By Theorem in  \cite{correctionMoriconeoftowers}, the toric limit $Y_{\tilde w}$ is not weak Fano.
Also we can see  $Z(\tilde w)$ is weak Fano but not
Fano (see \cite[Theorem 5.3]{Charyfano}).
\end{example}

\subsection{Local rigidity of BSDH varieties}\label{Rigidity}
In this section we obtain some vanishing results for the cohomology of tangent bundle of the
toric limit $Y_{\tilde w}$ and $Z(\tilde w)$.
 Let $T_X$ denote the tangent bundle of $X$, where $X=Y_{\tilde w}$ or $Z(\tilde w)$. Then we have 
 \begin{corollary}\label{vanishing1}\

  \begin{enumerate}
   \item 
If $\tilde w$ satisfies 
$I$, then  $H^i(Y_{\tilde w}, T_{Y_{\tilde w}})=0$ for all $i\geq 1$. In particular, $Y_{\tilde w}$ is locally rigid. 
 \item If $\tilde w$ satisfies 
 $I$, then  $H^i(Z(\tilde w), T_{Z(\tilde w)})=0$ for all $i\geq 1$. In particular, $Z(\tilde w)$ is locally rigid. 
   \end{enumerate}
 \end{corollary}
\begin{proof}
Proof of (1): 
If $\tilde w$ satisfies 
$I$, then by Lemma \ref{toriclimitfano}, $Y_{\tilde w}$ is a Fano variety.
By \cite[Proposition 4.2]{Bien1996}, since $Y_{\tilde w}$ is a smooth Fano toric variety, we get 
 $H^i(Y_{\tilde w}, T_{Y_{\tilde w}})=0 ~\mbox{for all}~  i\geq 1.$

Proof of (2):
From Theorem \ref{theorem1}, 
$\pi: \mathcal X \to \mathbb A^1$ is a smooth projective morphism 
 with $\mathcal X_0=Y_{\tilde w}$ and 
 $\mathcal X_{u}=Z(\tilde w)$ for $u\in \mathbb A^1$ , $u\neq 0$.  Hence (2) follows from (1)
  by semi-continuity theorem (see \cite[Theorem 12.8]{hartshorne}).
\end{proof}

\subsection{Log Fano BSDH varieties}

In \cite{anderson2014effective} and \cite{anderson2014schubert} log Fanoness of Schubert varieties and BSDH varieties were studied 
  respectively. Now we characterize the (suitably chosen) $\mathbb Q$-divisors $D'$ in $Z(\tilde w)$) for which $(Z(\tilde w), D')$ is log Fano.
 Recall that $Z_i=\{[(p_1,\ldots, p_r)]\in Z(\tilde w): p_i\in B\}$ is a divisor in $Z(\tilde w)$ (see Section \ref{preliminaries}).
 Let $\gamma_i=s_{\beta_r}\cdots s_{\beta_{i+1}}(\beta_i)$ for $1\leq i \leq r$.
Then, 
\begin{equation}\label{bi} \mathcal L_{\tilde w}(\delta)=\sum_{i=1}^rb_iZ_{i}  ~\mbox{with}~ 
 b_i=\langle \delta, \check \gamma_i \rangle=ht(\gamma_i),
\end{equation}
where $\delta$ is as in Section \ref{6.2BSDH} (see page 10),
$\mathcal L_{\tilde w}(\delta)$ is the homogeneous line bundle on $Z(\tilde w)$ corresponding to $\delta$ and  $ht(\beta)$ for a root $\beta=\sum_{i=1}^n n_i\alpha_i$, is the height defined by $ht(\beta)=\sum_{i=1}^nn_i$
(see \cite[Proof of Proposition 10]{mehta1985frobenius}).
When $\tilde w$ is reduced, $\gamma_i$ is a positive root and  
we can see the relation (\ref{bi}) from the 
Chevalley formula for intersection of  Schubert variety by a divisor
(see \cite[Page 410]{anderson2014schubert} or \cite{chevalley1994decompositions}).
It is known that  
  \begin{equation}\label{MR}
   -K_{Z(\tilde w)}=\sum_{i=1}^r(b_i+1)Z_{i}
     \end{equation}
(see \cite[Proposition 4]{mehta1985frobenius}).
     Let $D'=\sum_{i=1}^ra_iZ_{i}$ be a effective $\mathbb Q$-divisor in $Z(\tilde w)$, with 
$\lfloor D'\rfloor=0$, where  $\lfloor~\sum_i a_iZ_i\rfloor=\sum_i\lfloor a_i\rfloor Z_i$, $\lfloor x\rfloor$ is the greatest integer $\leq x$.
Then by (\ref{MR}), we get $$-(K_{Z(\tilde w)}+D')=\sum_{i=1}^r(b_i+1+a_i)Z_i.$$
For $1\leq i\leq r$, define 
$$f_i:=(b_i+1+a_i)-\sum_{\gamma_j\in \gamma_{P_i}(1)^+}c_j(b_j+1+a_j),$$
where $\gamma_{P_i}(1)^+:=\gamma_{P_i}(1)\cap \{\rho_l^+:1\leq l\leq r\}$
and $\gamma_{P_i}$ is the cone as in (\ref{4.1}) for the toric limit $Y_{\tilde w}$.

 Recall that if $X$ is smooth and $D$ is a normal crossing divisor, the pair $(X, D)$ is log Fano 
 if and only if $\lfloor D \rfloor=0$ and $-(K_X+D)$ is ample (see \cite[Lemma 2.30, Corollary 2.31 and Definition 2.34]{kollar2008birational}).

We prove,
\begin{corollary}\label{logfano02}
 The pair $(Z(\tilde w), D')$ is log Fano if $f_i>0$ for all $1\leq i\leq r.$
\end{corollary}
\begin{proof}
By definition of $D'$, the pair $(Z(\tilde w), D')$ is 
log Fano if and only if 
$-(K_{Z(\tilde w)}+D')$ is ample.
Now we prove $-(K_{Z(\tilde w)}+D' )$ is ample if  $f_i>0$ for all $1\leq i\leq r.$
Recall that $D_{\rho_i^+}$ is the divisor corresponding to $\rho_i^+\in \Sigma(1)$ and 
$ \mathcal {Z}_i^{x}= \pi^{-1}(x)\cap \mathcal Z_i$ for $x\in k$
(see Section \ref{preliminaries} and Section \ref{degeneration}).

By Theorem \ref{theorem1}, we have 
\begin{equation}\label{log1}
 \mathcal {Z}_i^{x}=Z_{i} ~\mbox{ for}~ x\neq 0 ~\mbox{ and}~
\mathcal Z_i^{0}=D_{\rho_i^+} .
\end{equation}

Assume that $f_i>0$ for all $1\leq i \leq r.$
 By (\ref{log1}) and by semi-continuity (see \cite[Theorem 1.2.7]{lazarsfeld2004positivity})
to prove $(Z(\tilde w), D')$ is log Fano it is enough to prove
$$\sum_{i=1}^r(b_i+1+a_i)D_{\rho_i^+}~\mbox{is ample}~.$$
By Lemma \ref{amplenef}, we see that $\sum_{i=1}^r(b_i+1+a_i)D_{\rho_i^+}$ is ample if and only if 
$$f_i=((b_i+1+a_i)-\sum_{\gamma_j\in \gamma_{P_i}(1)^+}c_j(b_j+1+a_j))>0 ~\mbox{for all} ~1\leq i \leq r.$$
Hence we conclude that $(Z(\tilde w), D')$ is log Fano. 
 \end{proof}

\section{Geometric properties of the toric limit}\label{AJPPK}
 
In this section we are going to recover the results of \cite{parameswaran2016toric} by using methods of toric geometry. 
In \cite{parameswaran2016toric}, they have assumed that $G$ is a simple algebraic group.
In our situation $G$ is a Kac-Moody group.
Recall the following:
\begin{enumerate}
 \item $\tilde w=s_{\beta_1}\cdots s_{\beta_{r}}$ and 
 $\tilde w'=s_{\beta_1}\cdots s_{\beta_{r-1}}$. 
\item The toric morphism 
$f_r:Y_{\tilde w} \to Y_{\tilde w'}$ is
induced by the lattice 
map 
$\overline f_r:\mathbb Z^r\to \mathbb Z^{r-1},$ 
the projection onto the first $r-1$ coordinates.
\end{enumerate}

As we discussed in Section \ref{degeneration}, there are two disjoint toric sections for the $\mathbb P^1$-fibration 
$f_r:Y_{\tilde w}\to Y_{\tilde w'}$ (see Lemma \ref{f_r}).
\begin{definition}\

\begin{enumerate}
 \item 
 
\underline{\bf Schubert and non-Schubert sections:}
We call the section corresponding to the maximal cone $\rho^+_r$ (respectively, $\rho_r^{-}$) in $\Sigma_F$  
(the fan of the fiber 
of $f_r$)  by  
\textquoteleft Schubert section $\sigma_{r-1}^0$' 
(respectively,  \textquoteleft non-Schubert section $\sigma_{r-1}^1$' ). \\
\item \underline{\bf Schubert point:} Let $\sigma\in \Sigma$ be the maximal cone generated by 
$\{e^+_1, \ldots, e^+_{r}\}.$
We call the point in $Y_{\tilde w}$ corresponding to the maximal cone $\sigma$  by \textquoteleft Schubert point'.\\

\item \underline{\bf Schubert line:}
We call the fiber of  $f_r$ over the Schubert point  by \textquoteleft Schubert line $L_r$'.

\end{enumerate}

\end{definition}

Note that these definitions agree with that of in \cite[Section 4]{parameswaran2016toric}.
Now onwards we denote $ \tilde w=(1,\ldots, r)$ (respectively, $\tilde w'=(1,\ldots, r-1)$) for the expression
$\tilde w=s_{\beta_1}\cdots s_{\beta_r}$ (respectively, $\tilde w'=s_{\beta_1}\cdots s_{\beta_{r-1}}$ ). 
Let $I=(i_1, \ldots , i_m)$ be a subsequence of $\tilde w$. Inductively we define the curve $L_{I}$ corresponding to $I$. 
Let 
$L_{I'}$ be the curve in $Y_{\tilde w'}$ corresponding to the
subsequence $I'=(i_1, \ldots, i_{m-1})$ of $I$ . Then define
$$L_{I}:=\sigma^1_{r-1}(L_{I'}) ~\mbox{and}~
 \sigma^0_{r-1}(L_{I'})={L_{I'}}.$$

Recall some more notations.
Let $X$ be a smooth projective variety, we define 
$$N_1(X)_{\mathbb Z}:=\{\sum_{\mbox{finite}}a_iC_i : a_{i}\in \mathbb Z , C_i ~\mbox{irreducible curve in } ~X\}/\equiv $$
where $\equiv$ is the numerical equivalence, i.e. $Z\equiv Z'$
if and only if $D\cdot Z=D\cdot Z'$ for all divisors $D$ in $X$.
We denote by $[C]$ the class of $C$ in $N_1(X)_{\mathbb Z}$.
Let $N_1(X):=N_1(X)_{\mathbb Z}\otimes \mathbb R$. 
It is a well known fact that $N_1(X)$ is a finite dimensional real vector space
dual to $N^1(X)$ (see \cite[Proposition 4, \S 1, Chapter IV]{kleiman1966toward}).
We have the following result:
\begin{lemma} The classes of Schubert lines $L_j$, $1\leq j\leq r$ form a basis of $N_1(Y_{\tilde w})$.
  \end{lemma}
\begin{proof} Proof is by induction on $r$. Assume that the result is true for $r-1$.
 Since $Y_{\tilde w}$ is a projective bundle over $Y_{\tilde w'}$ (see Lemma \ref{f_r}), then 
 by \cite[Lemma 1.1]{barton1971tensor}, 
  $$L_{r} ~\mbox{and}~ \sigma^0_{r-1}(L_{j}) ~\mbox{for}~ 1\leq j\leq r$$
 (the image of $L_{j}$ by the Schubert section in $Y_{\tilde w}$)
form a basis of $N_1(Y_{\tilde w})$. By definition of $L_{I}$, we have
$$\sigma^0_{r-1}(L_j)=L_{j} ~\mbox{for}~
1\leq j\leq r-1$$ and hence the result follows.
 \end{proof}

Let $1\leq j\leq r$.
Let $\mathscr D:=\{ e_l^{\epsilon_l}: 1\leq l\leq r  ~\mbox{and}~ \epsilon_l=+ ~\mbox{for all}~~ l\}$ .
Let $\mathscr D'_{j}:=\{e_l^{\epsilon_l}: 1\leq l\leq r  ~\mbox{and}~ \epsilon_l=+ ~\mbox{for all}~~ l\neq l; \epsilon_j=-\} $ .

 \begin{lemma}\label{LI}
 
   Fix $1\leq j\leq r$. Then the Schubert line $L_j$ is given by $$L_j=V(\tau_j)~,~\mbox{with}~ \tau_j=\sigma\cap \sigma'_j,$$ 
  intersection of two maximal cones in $\Sigma$, where $\sigma$ (respectively, $\sigma'_j$) is 
  generated by $\mathscr D$ (respectively, $\mathscr D'_j$).
   
\end{lemma}
 \begin{proof} 
 Let us consider the expression $\tilde w_j=s_{\beta_1}\cdots s_{\beta_{j}}$ for $1\leq j < r$. Let $\Sigma_j$ be the fan of the 
 toric variety $Y_{\tilde w_j}$. By Lemma \ref{f_r}, 
  $$f_j:Y_{\tilde w_{j}}\to Y_{\tilde w_{j-1}}$$ 
 is a $\mathbb P^1$-fibration 
induced by $\overline{f}_{j}:\mathbb Z^{j} \to \mathbb Z^{j-1}$ the
projection onto the first ${j-1}$ factors. Also note that the Schubert point in $Y_{\tilde w_{j-1}}$ corresponds to the maximal cone generated by 
$$\{e_l^+: 1\leq l\leq j-1\}$$ and the fan of the fiber is given by $\{e_{j}^{+}, 0, e_{j}^{-}\}.$ 
Let $\sigma_j$ (respectively, $\sigma_j'$) be the cone generated by 
$$\{e_{l}^+: 1\leq l \leq j\}$$ (respectively, 
$$\{e_{l}^+: 1\leq l \leq j-1\}\cup\{e_{j}^-\} \hspace{0.5cm} ).$$
Then by definition of Schubert line $L_j$, we can see that $L_j$ is the curve in $Y_{\tilde w_{j}}$  given by 
$$L_j=V(\tau_j), ~\mbox{where}~ \tau_j \in \Sigma_j
~\mbox{and}~ \tau_j= \sigma_j\cap \sigma'_j.$$ 
Since the Schubert section of $f_k$ for ($j\leq k\leq r$) corresponds to $e_k^+$, we see 
$$\sigma^0_{r}\circ \cdots \circ \sigma^0_{j+1}(L_j),$$ by abuse of notation we also denote it again by $L_j$ in $Y_{\tilde w}$,
is given by
$$L_j=V(\tau_j) ~\mbox{with}~\tau=\sigma \cap \sigma'_j,$$
where $\sigma$ and $\sigma'_j$ are as described in the statement. This completes the proof of the lemma.
\end{proof}

Let $\tau$ be a cone of dimension $r-1$ which is a wall, that is 
$\tau=\sigma\cap \sigma'$ for some $\sigma, \sigma'\in \Sigma$ of dimension $r$. 
 Let $\sigma$ (respectively, $\sigma'$) be generated by $\{u_{\rho_1}, u_{\rho_2}, \ldots , u_{\rho_r}\}$
(respectively, by $\{u_{\rho_2}, \ldots, u_{\rho_{r+1}} \}$) and
let $\tau$ be generated by $\{u_{\rho_2}, \ldots , u_{\rho_r}\}.$
 Then we get a linear relation,
 \begin{equation}\label{wallrelation}
  u_{\rho_1}+\sum_{i=2}^{r}b_iu_{\rho_i}+u_{\rho_{r+1}}=0
 \end{equation}
 The relation (\ref{wallrelation}) called {\it \bf wall relation} and we have 
 \begin{equation}\label{wall}
  D_{\rho}\cdot V(\tau)=\begin{cases}
                         b_i & ~\mbox{if}~ \rho=\rho_i ~\mbox{and}~ i\in \{2,3,\ldots, r\}\\
                         1 & ~\mbox{if}~ \rho=\rho_i ~\mbox{and}~ i\in \{1, r+1\}\\
                         0 & ~\mbox{otherwise}~
                        \end{cases}
 \end{equation}
(see \cite[Proposition 6.4.4 and eq. (6.4.6) page 303]{cox2011toric}). 
   We prove the following (see \cite[Proposition 33]{parameswaran2016toric}):
 \begin{proposition} Let $1\leq j\leq r$ and let $L_j$ be the Schubert line in $Y_{\tilde w}$. Then,
 $$K_{Y_{\tilde w}}\cdot L_j=-2-\sum_{k>j}\beta_{kj}.$$
  
 \end{proposition}

\begin{proof}

 By definition of $e_j^-$, we have 
 \begin{equation}\label{8.1}
e_j^++e_j^-+\sum_{k>j}\beta_{kj}e_k^+=0.
 \end{equation}
  By Lemma \ref{LI}, we have $L_j=V(\tau)$, with $\tau=\sigma\cap \sigma'$  where $\sigma$ 
 (respectively, $\sigma'$) is 
  generated by 
      $$\{e_l^{\epsilon_l}: 1\leq l\leq r, \epsilon_l=+ ~\mbox{for all}~~ l\} $$
  (respectively,  
    $$\{e_l^{\epsilon_l}: \epsilon_l=+ ~\mbox{for}~ 1\leq l\leq r ~\mbox{and}~ l\neq j, \epsilon_j=-\} \hspace{1cm} ).$$
  Hence (\ref{8.1}) is the {\it wall relation} for the curve $L_j$.
  Then by (\ref{wall}),
  we see that
    $$
  D_{\rho}\cdot L_j=\begin{cases}
                     1 & ~\mbox{if}~ \rho=\rho_j^+ ~\mbox{or}~ \rho_j^-.\\
                     \beta_{kj} & ~\mbox{if}~ \rho=\rho_k^+ ~\mbox{and}~ k>j.\\
                     0& ~\mbox{otherwise.}~
                    \end{cases}
$$
  Since $K_{Y_{\tilde w}}=-\sum_{\rho\in \Sigma(1)}D_{\rho}$, we get
  
  $$K_{Y_{\tilde w}}\cdot L_j=-2-\sum_{k>j}\beta_{kj}.$$
This completes the proof of the proposition.  
  \end{proof}

Now onwards we denote the subsequence $(i_1, \ldots, i_m)$ by $I_{i_1}$.
Let $\mathscr D''_{i_1}:=\{e_l^{\epsilon_l} : 1\leq l\leq r ~\mbox{and}~$
$$\epsilon_l=\begin{cases}
           + & ~\mbox{if}~  l\notin I_{i_1}\setminus \{i_1\} \\ 
         - & ~\mbox{if}~ l\in I_{i_1} 
         \end{cases}
\hspace{1.5cm} \} .$$

Let $\mathscr D'''_{i_1}:= \{e_l^{\epsilon_l} : 1\leq l\leq r ~\mbox{and}~$

$$\epsilon_l=\begin{cases}
           + & ~\mbox{if}~  l\notin I_{i_1} \\
         - & ~\mbox{if}~ l\in I_{i_1} 
         \end{cases}
\hspace{1.5cm} \} . $$

\begin{proposition}\label{LInew}
 The curve $L_{I_{i_1}}$ is given by 
 $$L_{I_{i_1}}=V(\tau_{i_1}) ~\mbox{with}~ \tau_{i_1}=\sigma_{i_1}\cap \sigma'_{i_1}, $$ where $\sigma_{i_1}$ (respectively, $\sigma'_{i_1}$)
   is the cone generated by $\mathscr D''_{i_1}$ 
 (respectively,   $\mathscr D'''_{i_1}$
).
\end{proposition}
\begin{proof} As in the proof of Lemma \ref{LI}, we start with $j=i_1$ and $L_{i_1}$ is the Schubert line in $Y_{\tilde w_{i_1}}$.
By Lemma \ref{LI}, we have 
$$L_{i_1}=V(\tau_{i_1})~\mbox{with}~\tau_{i_1}=\sigma_{i_1}\cap \sigma_{i_1}'.$$ 
By definition of $L_I$, we have 
$$\sigma^0_{i_2-1}\circ \cdots \circ \sigma^0_{i_1+1}(L_{i_1})=L_{i_1} ~\mbox{in}~ Y_{\tilde w_{i_2-1}}$$
and  
$$\sigma_{i_2}^1\circ \sigma^0_{i_2-1}\circ \cdots \circ \sigma^0_{i_1+1}(L_{i_1})=L_{\{i_1, i_2\}} ~\mbox{in}~ 
Y_{\tilde w_{i_2}}.$$
By repeating the process we conclude  that 
$$
L_{I_{i_1}}=V(\tau_{i_1}) ~\mbox{with}~ \tau_{i_1}=\sigma_{i_1}\cap \sigma'_{i_1},$$ where $\sigma_{i_1}$ and $\sigma'_{i_1}$ 
are as described in the statement. This completes the proof of the proposition.
\end{proof}

Recall
$NE(X)$ is the real convex cone  in $N_1(X)$ generated by classes of irreducible curves.
The  Mori cone  
 $\overline{NE}(X)$ is the closure of $NE(X)$ in $N_1(X)$ and it is a strongly  convex  rational polyhedral cone of maximal dimension (see for instance \cite[Chapter 6, page 293]{cox2011toric}).
Now we describe the Mori cone of the toric limit $Y_{\tilde w}$ in terms of the curves $L_{I_{i_j}}$'s defined above.
For this we need the following notation (see also \cite{Moriconeoftowers}).
Fix $1\leq i \leq r$.  Define:
  \begin{enumerate}
  \item Let 
$r\geq j> j_1=i\geq 1$ and define 
 $a_{1,j}:=\beta_{j_1j}$.
 \item Let $r\geq j_2>j_1$ be the least integer such that $a_{1,j}>0$, then define for $j>j_2$
 $$a_{2,j}:=\beta_{ij_2}\beta_{j_2j}-\beta_{ij}.$$
 \item Let $k>2$ and let $r\geq j_k>j_{k-1}$ be the least integer such that $a_{k-1, j}<0$, then inductively, define for $j>j_k$
  $$a_{k, j}:=-a_{k-1, j_k}\beta_{j_{k}j}+a_{k-1, j}.$$

 \item Let $\tilde {I_i}:=\{i=j_1,\ldots, j_m\}$. 
 \end{enumerate}
 
 \begin{example}\label{example1}
 Let $G=SL(5, \mathbb C)$ and let 
 $\tilde w=s_{\beta_1}\cdots s_{\beta_7}=s_{\alpha_2}s_{\alpha_1}s_{\alpha   _3}s_{\alpha_1}s_{\alpha_2}s_{\alpha_1}s_{\alpha_2}.$
  Let $i=1$. Then $j_1=1$ and
 \noindent (1) $a_{1, 2}=\beta_{12}=\langle \beta_2, \check \beta_1 \rangle =\langle \alpha_1, \check \alpha_2 \rangle =-1$  ;  (2)  $a_{1, 3}=\beta_{13}=\langle \beta_3, \check \beta_1 \rangle  =\langle \alpha_3, \check \alpha_2 \rangle =-1$ ;\\
(3)  $a_{1, 4}=\beta_{14}=\langle \beta_4, \check \beta_1 \rangle =\langle \alpha_1, \check \alpha_2 \rangle =-1$ ;
(4)
  $a_{1, 5}=\beta_{15}=\langle \beta_5, \check \beta_1 \rangle  =\langle \alpha_2, \check \alpha_2 \rangle = 2$ ;\\
(5)
 $a_{1, 6}=\beta_{16}=\langle \beta_6, \check \beta_1 \rangle =\langle \alpha_1, \check \alpha_2 \rangle =-1$ ;
 (6)
 $a_{1, 7}=\beta_{17}=\langle \beta_7, \check \beta_1 \rangle  =\langle \alpha_2, \check \alpha_2 \rangle =2$ .

Then by definition of $j_2$, we have $j_2=5$ and 
(1)   $a_{2, 6}=\beta_{15}\beta_{56}-\beta_{16}=\langle \beta_5, \check \beta_1 \rangle \langle \beta_6, \check \beta_5 \rangle-\langle \beta_{6}, \check \beta_{1} \rangle =\langle \alpha_1, \check \alpha_2 \rangle= -1 $ ;
  (2) $a_{2, 7}=\beta_{15}\beta_{57}-\beta_{17}=\langle \alpha_2, \check \alpha_2 \rangle = 2$ . 
 Then by definition of $j_3$, we have $j_3=6$ and 
  $a_{3, 7}=-a_{2, 6}\beta_{67}+a_{2, 7}=-(\langle \beta_6, \check \beta_5\rangle )(\langle \beta_7, \check \beta_6 \rangle )+ (\langle \beta_7, \check \beta_5\rangle )=-(-1)(-1)+(2)=1$.
 Therefore, we get $\tilde I_1=\{1, 5, 6\}$ .
 
  \end{example}

\begin{example}
 We use Example \ref{example1}, for $i=1$, we have $I_1=\{1, 5, 6\}$. Then 
 $$\mathscr D''_1=\{e_1^+, e_2^+, e_3^+, e_4^+, e_5^-, e_6^-, e_7^+\} ~\mbox{and}~ 
\mathscr D'''_1=\{e_1^-, e_2^+, e_3^+, e_4^+, e_5^-, e_6^-, e_7^+\}.$$

\end{example}

Fix $1\leq i \leq r$. Let 
      $$I_i:=\tilde I_i=\{i=j_1, j_2, \ldots, j_m\}$$  where $j_k$'s are as above.
With this notation we prove the following (see \cite[Theorem 22]{parameswaran2016toric}):
 \begin{theorem}\label{curves}
  The set  $\{L_{I_i}: 1\leq i\leq r\}$ of classes of curves forms a basis of $N_1(Y_{\tilde w})_{\mathbb Z}$ and every  torus invariant 
  curve
  in $N_1(Y_{\tilde w})$ lies in the cone generated by $\{L_{I_i}: 1\leq i\leq r\}$.
 \end{theorem}
\begin{proof}
By \cite[Proposition 4.16]{Moriconeoftowers}, for $1\leq i \leq r$ the curve $r(P_i)$ (see Section \ref{preliminaries} for the definition of $r(P_i)$) is given by 
$$r(P_i)=[V(\tau_i)],$$
where $\tau_i=\sigma_i\cap \sigma_i'$ and $\sigma_i$ (respectively, $\sigma'$) is generated by $\mathscr D_i''$ (respectively, $\mathscr D_i'''$).
From Proposition \ref{LInew}, we see that the class of the curve $L_{I_i}$ is $r(P_i)$ in $N_1(Y_{\tilde w})_{\mathbb Z}$.
By \cite[Theorem 4.7]{Moriconeoftowers}, we have $$\overline{NE}(Y_{\tilde w})=\sum_{i=1}^{r}\mathbb R_{\geq 0} r(P_i).$$ Also by \cite[Corollary 4.8]{Moriconeoftowers}, the set $\{r(P_i): 1\leq i\leq r\}$ forms a basis of $N_1(Y_{\tilde w})_{\mathbb Z}$.
Hence we conclude the assertion of the theorem.
\end{proof}

We recall some definitions: Let $V$ be a finite dimensional vector space over $\mathbb R$ and let $K$ be a (closed) cone in $V$.
A subcone $Q$ in $K$ is called extremal if $u, v \in  K, u + v \in Q$ then
$ u, v \in Q$.  A face of $K$ is an extremal subcone. A one-dimensional face is called an extremal ray. Note that an extremal ray is contained in the boundary of $K$.
 Then we have (see \cite[Theorem 30]{parameswaran2016toric}):
 \begin{corollary}\label{extre}
  The extremal rays of the toric limit $Y_{\tilde w}$ are precisely the curves $L_{I_i}$ for $1\leq i \leq r$.
 \end{corollary}
 \begin{proof}
  This follows from the proof of the Theorem \ref{curves}.
 \end{proof}

Let $X$ be a smooth projective variety. 
An extremal ray $R$ in the Mori cone $\overline {NE}(X)\subset N_1(X)$ is called Mori
if $R\cdot K_{X}<0$, where $K_X$ is the canonical divisor in $X$.
We have the following (see \cite[Theorem 35]{parameswaran2016toric}):
\begin{corollary}\label{mori}
  Fix $1\leq i \leq r$, 
  the class of curve $L_{I_i}$ is  Mori ray if and only if 
  either $|\gamma_{P_i}(1)|=0$, or $|\gamma_{P_i}(1)|=1$ with $c_j=1$ for $\gamma_j\in \gamma_{P_i}(1)$.
 \end{corollary}
\begin{proof}
Since $Y_{\tilde w}$ is a Bott tower (see Corollary \ref{f_r}), then the result follows from Proposition \ref{LInew} and \cite[Theorem 8.1]{Moriconeoftowers}.
\end{proof}

Now we have the following result for smooth projective toric varieties, which we can be deduced by using  Kleiman criterion and $NE(X)$ is a closed polyhedral cone. For completeness we give the proof here.

 \begin{lemma}\label{fanotoric}
Let $X$ be a smooth projective toric variety of dimension $r$.
Then $X$ is Fano if and only if every extremal ray is Mori. 
\end{lemma}
\begin{proof}
 By \cite[Theorem 6.3.20]{cox2011toric} (Toric Cone Theorem), we have 
 \begin{equation}\label{302}
  NE(X)=\overline{NE}(X)=\sum_{\tau\in \Sigma(r-1)}\mathbb R_{\geq 0}[V(\tau)].
 \end{equation}
 If $X$ is Fano, then by definition, $-K_X$ is ample.
By toric Kleiman criterion for ampleness \cite[Theorem 6.3.13]{cox2011toric}, we can see that $-K_X\cdot V(\tau)>0$ for all $\tau\in \Sigma(r-1)$.
Then $K_X\cdot V(\tau)<0$ for all $\tau\in \Sigma(r-1)$.
In particular, every extremal ray is Mori.

Conversely, let $\mathbb R_{\geq 0}[V(\tau)]$ be an extremal ray, by assumption
it is a Mori ray.
Then by definition of a Mori ray, we have $K_X\cdot V(\tau)<0$. 
This implies $-K_{X}\cdot V(\tau)>0$.
By (\ref{302}), $\overline{NE}(X)$ is a polyhedral cone and hence the extremal rays generate the cone $\overline{NE}(X)$.
Hence we see that $-K_{X}\cdot C>0$ for all classes of curves $[C]$ in $\overline{NE}(X)$.
Again by toric Kleiman criterion for ampleness, we conclude that $-K_X$ is ample and hence $X$ is Fano.
\end{proof}
 

Then we have the following (see \cite[Corollary 36]{parameswaran2016toric}):
\begin{corollary}\label{fano2}
 The toric limit $Y_{\tilde w}$ is Fano if and only if every extremal ray  in $\overline{NE}(Y_{\tilde w})$ is Mori.
\end{corollary}

{\bf Acknowledgements:} 
I would like to thank Michel Brion  for valuable discussions and
many critical 
comments.  Also, many thanks to the anonymous referees for their helpful suggestions.
\bibliographystyle{amsalpha}

\end{document}